\documentclass[a4paper,12pt]{amsart}
\usepackage{amssymb}
\usepackage{ifthen}
\usepackage{cite}
 \usepackage[dvips]{graphicx}
\nonstopmode \numberwithin{equation}{section}
\usepackage{amssymb}
\usepackage{ifthen}
\usepackage{graphicx}
\usepackage{amsmath}
\usepackage[T1]{fontenc} %skandit
\usepackage[utf8]{inputenc}
\usepackage[usenames,dvipsnames]{color}
\usepackage{color}
\usepackage[english]{babel}
\usepackage{fancyhdr}
\usepackage{fancybox}
\usepackage{tikz}
\nonstopmode \numberwithin{equation}{section}
\theoremstyle{plain}

\newtheorem{conj}{Conjecture}

\theoremstyle{definition}
\newtheorem{defn}{Definition}[section]

\newtheorem{thm}{Theorem}[section]
\newtheorem{thm-A}{Theorem A}[section]
\newtheorem{prob}{Problem}[section]
\newtheorem{cor}{Corollary}[section]
\newtheorem{ques}{Question}[section]
\newtheorem{prop}{Proposition}[section]
\newtheorem{rem}{Remark}[section]
\newtheorem{lem}{Lemma}[section]
\newtheorem{defi}{Definition}[section]

\theoremstyle{plain}
\newtheorem*{thmA}{Theorem A}
\newtheorem*{thmB}{Theorem B}
\newtheorem*{thmC}{Theorem C}
\newtheorem*{thmD}{Theorem D}
\newtheorem*{thmE}{Theorem E}
\newtheorem*{thmF}{Theorem F}
\newtheorem*{thmG}{Theorem G}
\newtheorem*{thmH}{Theorem H}
\newtheorem*{thmI}{Theorem I}
\newtheorem*{thmJ}{Theorem J}

\newtheorem*{lemA}{Lemma A}
\newtheorem*{lemB}{Lemma B}
\newtheorem*{lemC}{Lemma C}
\newtheorem*{lemD}{Lemma D}
\newtheorem*{lemE}{Lemma E}
\newtheorem*{lemF}{Lemma F}
\newtheorem*{lemG}{Lemma G}
\newtheorem*{lemH}{Lemma H}

\newcounter{minutes}
\divide\time by 60
\newcounter{hours}
\multiply\time by 60
\addtocounter{minutes}{-\time}
\newcounter {own}
\def\theown {\thesection       .\arabic{own}}

\newenvironment{pf}[1][]{%
 \vskip 3mm
 \noindent
 \ifthenelse{\equal{#1}{}}%
  {{\slshape Proof. }}%
  {{\slshape #1.} }%
 }%
{\qed\bigskip}

\newcounter{alphabet}

\newcommand{\real}{{\operatorname{Re}\,}}

\def\be{\begin{equation}}
\def\ee{\end{equation}}

\newcommand{\bee}{\begin{enumerate}}
\newcommand{\eee}{\end{enumerate}}

\newcommand{\blem}{\begin{lem}}
\newcommand{\elem}{\end{lem}}
\newcommand{\bthm}{\begin{thm}}
\newcommand{\ethm}{\end{thm}}
\newcommand{\bcor}{\begin{cor}}
\newcommand{\ecor}{\end{cor}}
\newcommand{\beg}{\begin{examp}}
\newcommand{\eeg}{\end{examp}}
\newcommand{\begs}{\begin{examples}}
\newcommand{\eegs}{\end{examples}}

\newcommand{\bdefn}{\begin{defn}}
\newcommand{\edefn}{\end{defn}}

\newcommand{\bprob}{\begin{prob}}
\newcommand{\eprob}{\end{prob}}
\newcommand{\bei}{\begin{itemize}}
\newcommand{\eei}{\end{itemize}}

\newcommand{\bcon}{\begin{conj}}
\newcommand{\econ}{\end{conj}}
\newcommand{\bcons}{\begin{conjs}}
\newcommand{\econs}{\end{conjs}}
\newcommand{\bprop}{\begin{prop}}
\newcommand{\eprop}{\end{prop}}
\newcommand{\br}{\begin{rem}}
\newcommand{\er}{\end{rem}}
\newcommand{\brs}{\begin{rems}}
\newcommand{\ers}{\end{rems}}
\newcommand{\bo}{\begin{obser}}
\newcommand{\eo}{\end{obser}}
\newcommand{\bos}{\begin{obsers}}
\newcommand{\eos}{\end{obsers}}
\newcommand{\bpf}{\begin{pf}}
\newcommand{\epf}{\end{pf}}
\newcommand{\ba}{\begin{array}}
\newcommand{\ea}{\end{array}}
\newcommand{\beq}{\begin{eqnarray}}
\newcommand{\beqq}{\begin{eqnarray*}}
\newcommand{\eeq}{\end{eqnarray}}
\newcommand{\eeqq}{\end{eqnarray*}}

\begin{document}

\title{A Unified Study of Bohr's Inequality for analytic and harmonic mappings on the Unit Disk}

\author{Molla Basir Ahamed}
\address{Molla Basir Ahamed, Department of Mathematics, Jadavpur University, Kolkata-700032, West Bengal, India.}
\email{mbahamed.math@jadavpuruniversity.in}

\author{Partha Pratim Roy}
\address{Partha Pratim Roy, Department of Mathematics, Jadavpur University, Kolkata-700032, West Bengal, India.}
\email{pproy.math.rs@jadavpuruniversity.in}

\author{Sujoy Majumder}
\address{Department of Mathematics, Raiganj University, Raiganj, West Bengal-733134, India.}
\email{sm05math@gmail.com}

\subjclass[{AMS} Subject Classification:]{Primary 30A10, 30H05, 30C35, Secondary 30C45}

\def\thefootnote{}
\footnotetext{ {\tiny File:~\jobname.tex,
printed: \number\year-\number\month-\number\day,
          \thehours.\ifnum\theminutes<10{0}\fi\theminutes }
} \makeatletter\def\thefootnote{\@arabic\c@footnote}\makeatother
\begin{abstract} 
We investigate improved forms of the Bohr inequality, using the quantity $S_r/\pi$, for analytic self-maps in class $\mathcal{B}$ of $\mathbb{D}$, where $S_r$ is the area measure of $\mathbb{D}_r$. We then generalize the inequality for harmonic mappings ($\mathcal{P}^0_{\mathcal{H}}(M)$ and $\mathcal{W}^0_{\mathcal{H}}(\alpha)$ of the form $f = h + \overline{g}$) by introducing a sequence $\{\varphi_n(r)\}_{n=0}^\infty$ of differentiable, increasing functions on $[0, 1)$. The Hurwitz–Lerch Zeta function is utilized for some consequences, and all results are shown to be sharp.
\end{abstract}
\maketitle
\pagestyle{myheadings}
\markboth{M. B. Ahamed, P. P. Roy, and S. Majumder}{The Bohr inequalities for analytic and harmonic mappings}
\tableofcontents
\section{\bf Introduction}
Let $ \mathbb{D}:=\{z\in \mathbb{C}:|z|<1\} $ denote the open unit disk in the complex plane. In 1914, Harald Bohr (see \cite{Bohr-1914}) established that for the class $H_\infty$ of all bounded analytic functions $f$ on $\mathbb{D}$, the following inequality holds.
Let $ \mathbb{D}:=\{z\in \mathbb{C}:|z|<1\} $ denote the open unit disk in $ \mathbb{C} $. Harald Bohr in $1914$ (see \cite{Bohr-1914}) stated that if $H_\infty$ denotes the class of all bounded analytic functions $f$ on $\mathbb{D}$, then the following inequality holds
\begin{align}\label{Eq-1.1}
	B_0(f,r):=|a_0|+\sum_{n=1}^{\infty}|a_n|r^n\leq||f||_\infty:=\displaystyle\sup_{z\in\mathbb{D}}|f(z)| \;\;\mbox{for}\;\; 0\leq r\leq\frac{1}{3},
\end{align}
where $a_k=f^{(k)}(0)/k!$ for $k\geq 0$. The constant $1/3$ is called the Bohr radius and the inequality in \eqref{Eq-1.1} is called Bohr inequality for the class $\mathcal{B}$ of analytic self-maps on $\mathbb{D}$. Henceforth, if there exists a positive real number $r_0$ such that an inequality of the form \eqref{Eq-1.1} holds for every elements of a class $\mathcal{M}$ for $0\leq r\leq r_0$ and fails when $r>r_0$, then we shall say that $r_0$ is an sharp bound for $r$ in the inequality w.r.t. to class $\mathcal{M}$.\vspace{1.2mm}

Bohr initially established that the inequality \eqref{Eq-1.1} is valid for all $|z|\leq1/6$. Subsequently, M. Riesz, I. Schur, and F. W. Wiener independently extended this result, verifying the inequality for the broader interval $0\leq r\leq1/3$. It has since been confirmed that the bound $1/3$ is sharp, a fact deduced by examining specific examples within the family of conformal automorphisms of the unit disk \( \mathbb{D} \).
\vspace{1.2mm} 

 Bohr's inequality, initially linked to the convergence of Dirichlet series by Harald Bohr, has evolved into a key topic in modern function theory (see \cite{Boas-1997, Bhowmik-2018, Chen-Liu-Pon-RM-2023}). It plays a pivotal role in characterizing Banach algebras satisfying the von Neumann inequality (see \cite{Dixon & BLMS & 1995}), with several distinct proofs and extensions, including the multidimensional Bohr radius \( \mathcal{K}^{\infty}_n \) introduced by Boas and Khavinson. Recent works (see \cite{Ponnusamy-RM-2020, Ponnusamy-HJM-2021, Ponnusamy-JMAA-2022, Chen-Liu-Pon-RM-2023, Kumar-CVEE-2023}) have extended Bohr's inequality to more general power series by replacing the classical basis \( \{r^n\} \) with sequences \( \{\varphi_n(r)\} \) of non-negative continuous functions, requiring locally uniform convergence of their sum on \( [0,1) \). Additionally, studies such as \cite{Ahamed-AASFM-2022, Ponnusamy-RM-2021} investigated Bohr phenomena on shifted disks \( \Omega_\gamma \) (\( 0 \leq \gamma < 1 \)), while \cite{Allu-JMAA-2021} explored Bohr radii for subclasses of starlike and convex functions in \( \mathbb{D} \).
  \vspace{1.2mm}
 
 The Bohr inequality for holomorphic and pluriharmonic mappings with values in complex Hilbert spaces are obtained in \cite{hamada-MN-2023}. However, it can be noted that not every class of functions has Bohr phenomenon, for example, B\'en\'eteau \emph{et al.} \cite{Beneteau-2004} showed that there is no Bohr phenomenon in the Hardy space $ H^p(\mathbb{D},X), $ where $p\in [1,\infty).$ In \cite{Liu-Liu-JMAA-2020}, Liu and Liu have shown that Bohr's inequality fails to hold for the class $ \mathcal{H}(\mathbb{D}^2, \mathbb{D}^2) $ of holomorphic functions $ f : \mathbb{D}^2\rightarrow \mathbb{D}^2 $ having lacunary series expansion. Over the past few years, Bohr's inequality has gained a significant attention, resulting in extensive research and extensions in various directions and settings (see \cite{Allu-CMFT,Boas-2000,Bhowmik-CMFT-2019,Boas-1997,Paulsen-PLMS-2002,Paulsen-BLMS-2006,Ponnusamy-CMFT-2020,Tomic-1962,Lata-Singh-PAMS-2022,Liu-JMAA-2021,Arora-Kumar-Ponnusamy- Monatsh Math-2025, Dai-BMMSS-2025,Kumar-NYJ-2025, Hamada-AMP-2025,Hamada-BMMSS-2024} and references therein). \vspace{1.2mm}
 
 In this paper, we have two aims. First, we aim to establish improved  Bohr inequalities for the class $\mathcal{B}$ incorporating the term $S_r/(\pi-S_r)$. Second, we aim to obtain generalized Bohr inequality  for a certain class of harmonic mappings, using a sequence of differentiable functions. The paper is organized as follows: In Section 2, we study the improved Bohr inequality for the class $\mathcal{B}$ of bounded analytic functions on the unit disk $\mathbb{D}$. In Section 3, we discuss the generalization of the Bohr inequality for certain classes of harmonic mappings and its applications. The proof of the main results are discussed in each section individually.
\section{\bf Bohr inequalities for the class $\mathcal{B}$ of bounded analytic functions on unit disk $ \mathbb{D} $} 
Before we proceed with our discussion, let us introduce some notations.
 \subsection{\bf Basic Notations}  Let
 \begin{align*}
 	\mathcal{B}=\{f\in H_{\infty} \;:\;||f||_{\infty}\leq 1\}.
 \end{align*} Also, for $f(z)=\sum_{n=0}^{\infty}a_nz^n\in \mathcal{B}$ and $f_0(z):=f(z)-f(0)$, we let for convenience 
 \begin{align*}
 	B_k(f,r):=\sum_{n=k}^{\infty}|a_n|r^n \;\mbox{for }k\geq 0,\; ||f_0||^2_r:=\sum_{n=1}^{\infty}|a_n|^2r^{2n} \;\mbox{and} \;||f_1||^2_r:=\sum_{n=2}^{\infty}|a_n|^2r^{2n}
 \end{align*}
so that
\begin{align*}
\begin{cases}
	 B_0(f,r)=|a_0|+B_1(f,r),\\   B_0(f,r)=|a_0|+|a_1|r+B_2(f,r),\\ B_0(f,r)=|a_0|+|a_1|r+|a_2|r^2+B_3(f,r), \; \mbox{so on}.
\end{cases}
 \end{align*}
 Further, we define the quantity $A(f_0,r)$ by 
 \begin{align*}
 	\begin{cases}
 		A(f_0,r):=\left(\dfrac{1}{1+|a_0|}+\dfrac{r}{1-r}\right)||f_0||^2_r,\vspace{2mm}\\
 		A(f_1,r):=\left(\dfrac{1}{1+|a_0|}+\dfrac{r}{1-r}\right)||f_1||^2_r.
 	\end{cases}
 \end{align*}
 Now we concentrate on another well-known result by W.~Rogosinski~\cite{Rogosinski-1923}, which states that for any holomorphic self-mapping \( f(z) = \sum_{n=0}^{\infty} a_n z^n \) of the open unit disk \( \mathbb{D} \),
 \begin{align*}
 	\bigg|\sum_{n=0}^{N}a_nz^n\bigg|\leq 1
 \end{align*}
 for \( |z| = r \leq 1/2 \) and for all \( N \geq 1 \). This quantity \( 1/2 \) is the best possible. Further progress on this problem has been made in~\cite{Aizenberg-2005,Aizenberg-2012,Nabetani-1935}. It is worth mentioning that many questions (including the refined forms) on a related new concept called the Bohr-Rogosinski phenomenon are currently being studied (cf.~\cite{Huang-Liu-Ponnusamy-2020,Liu-Liu-Ponnusamy-2021}). In fact, the Bohr--Rogosinski radius is analogous to the Bohr radius, which has been defined (see~\cite{Rogosinski-1923}) as follows: If \( f \in \mathcal{B} \) is given by
 \begin{align*}
 	 f(z) = \sum_{n=0}^{\infty} a_n z^n,
 \end{align*}
 then for an integer  $N \geq 1$ , we have \( |S_N(z)| < 1 \) in the disk \( \{ z \in \mathbb{C} : |z| < 1/2 \} \), and the radius \( 1/2 \) is best possible, where $ S_N(z) = \sum_{n=0}^{N} a_n z^n
 $ denotes the \( N \)-th partial sum of \( f \). The radius $r = 1/2$ is called the \emph{Rogosinski radius}. For \( f \in \mathcal{B} \) given by \( f(z) = \sum_{n=0}^{\infty} a_n z^n \), the \emph{Bohr--Rogosinski sum} \( R_f^N(z) \) of \( f \) is defined by
 \begin{align}\label{EQN-2.1}
 	R_f^N(z)=|f(z)|+\sum_{n=N}^{\infty}|a_n|r^n,\;\;\;|z|=r
 \end{align}
 It is worth noticing that for \( N = 1 \), the inequality \eqref{EQN-2.1} is related to the classical Bohr sum, in which \( |f(0)| \) is replaced by \( |f(z)| \). The inequality $R_f^N(z) \leq 1$ is called the Bohr Rogosinski inequality. If \( B \) and \( R \) denote the Bohr radius and Bohr Rogosinski radius, respectively, if \( B \) and \( R \) denote the Bohr radius and Bohr--Rogosinski radius, respectively, then it is easy to see that \( B = {1}/{3} < {1}/{2} = R \).\vspace{1.5mm}
 
 Kayumov and his co-authors (see \cite{Ponnusamy-2017,Kayu-Kham-Ponnu-2021-JMAA}) derived the Bohr-Rogosinski inequality  for class $ \mathcal{B} $ by investigating Rogosinski's inequality and Rogosinski's radius, as presented in \cite{Rogosinski-1923}.
\begin{thmA}\emph{\cite[Corolary 1, p.4]{Kayu-Kham-Ponnu-2021-JMAA}}\label{th-1.9}
	 Suppose that $ f(z)=\sum_{n=0}^{\infty}a_nz^n\in\mathcal{B} $. Then 
\begin{align*}
	 |f(z)|+B_N(f,r)\leq 1\;\;\mbox{for}\;\; r\leq R_N,
\end{align*}
where $ R_N $ is the positive root of the equation $ 2(1+r)r^N-(1-r)^2=0 $. The radius $ R_N $ is the best possible. Moreover, 
\begin{align*}
		|f(z)|^2+B_N(f,r)\leq 1\;\;\mbox{for}\;\; r\leq R^{\prime}_N,
\end{align*}
where $ R^{\prime}_N $ is the positive root of the equation $ (1+r)r^N-(1-r)^2=0 $. The radius $ R^{\prime}_N $ is the best possible.
\end{thmA}
 In recent years, Theorem A  generates a significant amount of research activity on Bohr-Rogosinki inequalities for different classes of functions. However, Liu \emph{et al.} \cite{Liu-Shang-Xu-JIA-2018} have proved the following result for functions in the class $ \mathcal{B} $, where the coefficients $ |a_0| $ and $ |a_1| $ in majorant series are replaced by $ |f(z)| $ and $ |f^{\prime}(z)| $, respectively.
\begin{thmB}\emph{\cite[Theorem 2.1]{Liu-Shang-Xu-JIA-2018} }\label{th-1.6}
Suppose that $ f(z)=\sum_{n=0}^{\infty}a_nz^n\in\mathcal{B} $. Then 
\begin{align*}
|f(z)|+|f^{\prime}(z)|r+B_2(f,r)\leq 1\;\; \mbox{for}\;\; r\leq(\sqrt{17}-3)/4.
\end{align*}
The constant $ (\sqrt{17}-3)/4 $ is best possible.
\end{thmB}
Further, a refined Bohr inequality with a suitable setting was investigated in \cite{Liu-Liu-Ponnusamy-2021}, where the following result was obtained.
 \begin{thmC}\label{Th-2.3} \emph{\cite[Theorem 2]{Liu-Liu-Ponnusamy-2021}}
	Suppose that $ f(z)\in\mathcal{B} $ and $f(z)=\sum_{n=0}^{\infty}a_nz^n$. Then
	\begin{align}\label{Eq-2.1}
	|f(z)|+B_1(f,r)+A(f_0,r)\leq 1 
	\end{align}
	for $ |z|= r_0 \leq(2)/(3+|a_0|+\sqrt{5}(1+|a_0|))$ and the constant $r_0$ is the best possible and $r_0>\sqrt{5}-2$. Moreover, 
	\begin{align}\label{Eq-2.2}
		|f(z)|^2+B_1(f,r)+A(f_0,r)\leq 1,   
	\end{align} 
	for $|z|= r \leq r^{\prime}_0,$ where $r^{\prime}_0$ is the unique root in $(0, 1)$ of the equation
	\begin{align*}
		(1-|a_0|^3)r^3-(1+2|a_0|)r^2-2r+1=0
	\end{align*}
	and $r^{\prime}_0$ is the best possible and $1/3<r^{\prime}_0<1/(2+|a_0|)$. 
\end{thmC}
Moreover, Liu \emph{et. al} (see \cite{Liu-Liu-Ponnusamy-2021}) refined the inequalities (\ref{Eq-2.1}) and (\ref{Eq-2.2}) furthur by replacing the second coefficient $|a_1|$ in the majorant series $ B_0(f, r)$ by the quantity $|f^{\prime}(z)|$ and obtained the following result.
\begin{thmD} \emph{\cite[Theorem 7]{Liu-Liu-Ponnusamy-2021}}
	Suppose that $ f(z)\in\mathcal{B} $ and $f(z)=\sum_{n=0}^{\infty}a_nz^n$. Then
	\begin{align}\label{Eqnnn-2.3}
		|f(z)|+|f^{\prime}(z)|r+B_2(f,r)+A(f_0,r)\leq 1 
	\end{align}
	for $ |z|= r \leq(\sqrt{17}-3)/4$ and the constant $(\sqrt{17}-3)/4$ is the best possible. Moreover, 
	\begin{align}\label{Eqnnn-2.4}
	&|f(z)|^2+|f^{\prime}(z)|r+B_2(f,r)+A(f_0,r)\leq 1 
	\end{align}
for $ |z|= r \leq r_0 $, where $ r_0\approx 0.385795 $ is the unique positive root of the equation $ 1-2r-r^2-r^3-r^4=0 $	and the constant $r_0$ is the best possible. 
\end{thmD} 

\subsection{\bf Improved Bohr inequality for the class $\mathcal{B}$}
 Let $f$ be holomorphic in $\mathbb D$, and for $0<r<1$,  let $\mathbb D_r=\{z\in \mathbb C: |z|<r\}$.
Throughout the paper,  $S_r:=S_r(f)$ denotes the planar integral
\begin{align*}
	S_r:=\int_{\mathbb D_r} |f'(z)|^2 d A(z).
\end{align*}
The quantity $S_r$ plays a significant role in the study of improved Bohr inequalities for the class $\mathcal{B}$. Kayumov and Ponnusamy \cite{Ponnusamy-2017} proved  the following sharp inequality  for functions in the class 
$ \mathcal{B}$, for the function $f(z)=\sum_{n=0}^{\infty}a_nz^n$,  
\begin{align}\label{Eqn-2.3}
	\frac{S_r}{\pi}:= \sum_{n=1}^\infty n|a_n|^2 r^{2n}\leq r^2\frac{(1-|a_0|^2)^2}{(1-|a_0|^2r^2)^2}\; \mbox{for}\;0<r\leq1/\sqrt{2}.
\end{align}

The investigation into improving the Bohr inequalities and their different sharp versions is a thriving area of research in the Bohr phenomenon. In fact, mathematicians have formulated several improved versions of the Bohr inequalities by incorporating the quantities $S_r/\pi$ and $S_r/(\pi-S_r)$ in the Bohr inequalities. This line of inquiry was initiated by Kayumov and Ponnusamy (see \cite[Theorem 1]{Kayumov-CRACAD-2018}), who established improved versions of the Bohr inequalities, denoted as $B_0(f,r)$ and $|a_0|^2+B_1(f,r)$, by incorporating $S_r/\pi$ and $S_r/(\pi-S_r)$, respectively. Subsequently, Ismagilov \emph{et al.} (see \cite[Theorem 3]{Ismagilov-2020-JMAA}) further improved the primary inequality of Kayumov and Ponnusamy (refer to \cite[Theorem 1]{Kayumov-CRACAD-2018}) by substituting the initial coefficient $|a_0|$ with $|f(z)|$. Following this, Liu \emph{et al.} (see \cite[Theorem 4]{Liu-Liu-Ponnusamy-2021}) not only refined both inequalities introduced by Kayumov and Ponnusamy (in \cite[Theorem 1]{Kayumov-CRACAD-2018}) but also included an augmenting term $A(f_0, r)$ into the inequalities, thereby ushering in a significant advancement in this domain of study.\vspace{1.2mm}

 Ismagilov \textit{et al.} (see \cite{Ismagilov-2021-JMS}) further investigated the classical Bohr inequality and obtained the following sharp result, replacing the quantity $S_r/\pi$ with $S_r/(\pi-S_r)$.
\begin{thmE} \emph{\cite[Theorems 10, 11]{Ismagilov-2021-JMS}}
	Suppose that $ f(z)=\sum_{n=0}^{\infty}a_nz^n\in\mathcal{B} $. Then 
	\begin{align*}
		B_0(f,r)+ \frac{16}{9} \left(\frac{S_{r}}{\pi-S_r}\right) \leq 1 \quad \mbox{for} \quad r \leq \frac{1}{3},
	\end{align*}
	and the number $16/9$ cannot be improved. Moreover, 
	\begin{align*}
		|a_{0}|^{2}+B_1(f,r)+ \frac{9}{8} \left(\frac{S_{r}}{\pi-S_r}\right) \leq 1 \quad \mbox{for} \quad r \leq \frac{1}{2},
	\end{align*}
	and the number $9/8$ cannot be improved.
\end{thmE}
 Very recent, an  improvement involves substituting the coefficient $|a_0|$ with $|f(z)|^2$ and using the quantity of $S_r/(\pi-S_r)$ instead of $S_r/\pi$, the following result is obtained.
\begin{thmF}\emph{\cite[Theorem 2.5]{Ahamed-CMB-2023}}
Suppose that $ f(z)\in\mathcal{B} $ and $f(z)=\sum_{n=0}^{\infty}a_nz^n$. Then
\begin{align}\label{Eq-3.9}
	|f(z)|^2+B_1(f,r)+A(f_0, r)+\lambda\left(\frac{S_r}{\pi-S_r}\right)\leq 1, 
\end{align} for $|z|= r \leq 1/3,$ and the constant $\lambda=8/9$ cannot be improved.
\end{thmF}
Therefore, drawing motivation from the recent results (see \cite[Theorem 2.1]{Ahamed-CMFT1-2022}, \cite[Theorem 10, 11]{Ismagilov-2021-JMS}, \cite[Theorem 2.5]{Ahamed-CMB-2023}) and using the inequality $S_r/(\pi-S_r)$, we note that the parameter $\lambda$ remained unchanged in all previous discussions when $S_r/(\pi-S_r)$ was used instead of $S_r/\pi$. To further explore the subject, the following question naturally arises:
\begin{ques}\label{Q-3.4}
Is it possible to refine the inequality~(\ref{Eqnnn-2.3}), originally established by Liu \textit{et al.} (and discussed in~\cite{Liu-Liu-Ponnusamy-2021}), by incorporating the term ${S_r}/{(\pi - S_r)}$? If so, what can be said about the sharpness of this improvement?
\end{ques}
To give a complete answer to Question \ref{Q-3.4}, we establish the following result, which shows that the obtained radius cannot be improved further.
\begin{thm}\label{th-2.2}
Suppose that $ f(z)=\sum_{n=0}^{\infty}a_nz^n\in\mathcal{B} $ and $S_r$ denotes the Riemann surface of the function $f^{-1}$ defined on the image of the sub-disk $|z|\leq r$ under the mapping $f$. Then 
\begin{align}\label{Eq-2.16}
\mathcal{L}_{1,f}(r)&:=|f(z)|+|f^\prime(z)|r+B_2(f,r)+A(f_0, r)+\lambda\left(\frac{S_r}{\pi-S_r}\right)\leq 1
\end{align}
for $ r\leq(\sqrt{17}-3)/4 $, where the numbers  $\lambda=(221-43\sqrt{17})/64$ and $(\sqrt{17}-3)/4\;$ are sharp.
\end{thm}
 \begin{rem}
 Furthermore, the inequality \eqref{Eqnnn-2.4} by Liu \textit{et al.}~\cite{Liu-Liu-Ponnusamy-2021} does not permit improvement via the inclusion of $\frac{S_r}{\pi - S_r}$ as the resulting Bohr radii lie outside $(0, 1)$.
 \end{rem}

\subsection{\bf Key Lemmas and Proof of Theorem \ref{th-2.2}}
In this section, we present some necessary lemmas which will play key roles to prove Theorem \ref{th-2.2} for the class $\mathcal{B}$.

\begin{lemA}(\cite[Lemma 4]{Liu-Liu-Ponnusamy-2021})
	Suppose that $ f(z)=\sum_{n=0}^{\infty}a_nz^n\in\mathcal{B} $. Then for any N$\in\mathbb{N}$, the following inequality holds:
	\begin{align*}
		\sum_{n=N}^{\infty}&|a_n|r^n+{\rm sgn}(t)\sum_{n=1}^{t}|a_n|^2\dfrac{r^N}{1-r}+\left(\dfrac{1}{1+|a_0|}+\dfrac{r}{1-r}\right)\sum_{n=t+1}^{\infty}|a_n|^2r^{2n}\\&\leq (1-|a_0|^2)\dfrac{r^N}{1-r}, 
	\end{align*}
	\;\;\mbox{for}\;\; $ r\in[0,1),$ where $t=\lfloor{(N-1)/2}\rfloor$.
\end{lemA}
\begin{lemB} (\cite[p. 35]{Graham-2003})
	Suppose that $ f(z)=\sum_{n=0}^{\infty}a_nz^n\in\mathcal{B} $. Then  $|a_n|\leq 1-|a_0|^2$ \;\; \mbox{for all}\;\; $n=1,2,...$
\end{lemB}
\begin{lemC} (\cite{Dai-Proc-2008},\cite[Theorem 2]{St. Ruscheweyh})
	Suppose that $ f(z)=\sum_{n=0}^{\infty}a_nz^n\in\mathcal{B} $. Then, for all $k=1,2,3.....$, we have
	\begin{align*}
		|f^{(k)}(z)|\leq \frac{k!(1-|f(z)|^2)}{(1-|z|^2)^k}(1+|z|)^{k-1}, \;\; |z|<1.
	\end{align*}
	\end{lemC}

\begin{proof}[\bf Proof of Theorem \ref{th-2.2}]
We suppose that $|a_0|=a \in (0,1)$. By combining Lemma A (with $N=2$) with the inequality for $|f(z)|$ and the Schwarz-Pick lemma, we obtain
\begin{align*}
\mathcal{L}_{1,f}(r)&:=|f(z)|+|f^\prime(z)|r+B_2(f,r)+A(f_0, r)+\lambda\left(\frac{S_r}{\pi-S_r}\right)\\
&\leq \frac{r+a}{1+ra}+\frac{r}{1-r^2}\left(1-\left(\frac{r+a}{1+ra}\right)^2\right)+\frac{(1-a^2)r^2}{1-r}+\lambda\frac{(1-a^2)^2r^2}{(1-r^2)(1-a^4r^2)}\\&=\frac{r+a}{1+ra}+\frac{(1-a^2)r}{(1+ar)^2}+\frac{(1-a^2)r^2}{1-r}+\lambda\frac{(1-a^2)^2r^2}{(1-r^2)(1-a^4r^2)}\\&:=A_1(r)\;\;\mbox{for}\;\;0<r<1.
\end{align*}
It is easy to see that $A_1(r)$ is an increasing function of $r$; thus, it suffices to show that
\begin{align*}
 A_1\left((\frac{\sqrt{17}-3)}{4}\right)&=\frac{(-3+\sqrt{17})+4a}{4+\left(-3+\sqrt{17}\right)a}+\frac{(-3+\sqrt{17})^2(1-a^2)}{4(7-\sqrt{17})}+\frac{4(-3+\sqrt{17})(1-a^2)}{\left(4+(-3+\sqrt{17})a\right)^2}\\&\quad+\frac{(-3+\sqrt{17})^2(221-43\sqrt{17})(1-a^2)^2}{16\left(-5+3\sqrt{17})\right)(8+(-13+3\sqrt{17})a^4)}\leq1,
\end{align*} 
which is equivalent to
\begin{align*}
A_1\left((\frac{\sqrt{17}-3)}{4}\right)-1=	\frac{(1-a)^3F_1(a)}{2(4+(-3+\sqrt{17})a)^2(52\sqrt{17}-172+(611-149\sqrt{17})a^4)}\leq0,
\end{align*}
where
	\begin{align*}
		F_1(a)&=8(-3445+851\sqrt{17})+16\left(-195+53\sqrt{17}\right)a+2\left(24191-5849\sqrt{17}\right)a^2\\&\quad+2\left(-3927+961\sqrt{17}\right)a^3+8\left(-5981+1451\sqrt{17}\right)a^4\\&\quad+8\left(-3187+773\sqrt{17}\right)a^5.
	\end{align*}
	It is clear that, $A_1((\sqrt{17}-3)/4)-1\leq0$ because when, $a\rightarrow1^{-}$, we see that,
	\begin{align*}
		F_1(0)=(80-48\sqrt{17}+4(-109+27\sqrt{17})a^4)=-108.61366...<0\\\;\;\mbox{and}\;\; F_3(a)\rightarrow 560.65568...>0\; \mbox{as a}\to 1^{-}.
	\end{align*} 
	Consequently, we have $A_1\left(({\sqrt{17}-3})/{4}\right)\leq 1$, it means that the desired inequality \eqref{Eq-2.16} is correct for $r\leq({\sqrt{17}-3})/{4}$.\vspace{1.2mm} 

To show the constant $\lambda=(221-43\sqrt{17})/64$ is best possible, let us consider the function $f^{*}_a$ which is given by 
\begin{align}\label{EQ-4.2}
	f^*_a(z):=\frac{a+z}{1+az}=a+(1-a^2)\sum_{k=1}^{\infty}(-a)^{k-1}z^k,\;\;z\in\mathbb{D},\;\;\mbox{where}\;\;a\in(0,1).
\end{align}
 We then choose the specific value $z=r=(\sqrt{17}-3)/4$ and a small perturbation $\epsilon>0$. Using these choices, one can check that
	\begin{align*}
		\mathcal{L}_{1,f^{*}_a}(r)&:=|f^{*}_a(r)|+|(f^{*}_a)^\prime(r)|r+B_2(f^{*}_a,r)+A(({f}^*_a)_0,r)+\bigg(\frac{221-43\sqrt{17}}{64}+\epsilon\bigg)\left(\frac{S_r}{\pi-S_r}\right)\\&=\frac{r+a}{1+ra}+\frac{(1-a^2)r}{(1+ar)^2}+\frac{(1-a^2)ar^2}{(1-ar)}+\frac{(1+ar)}{(1+a)(1-r)}\frac{(1-a^2)^2r^2}{1-a^2r^2}\\&\quad+\bigg(\frac{221-43\sqrt{17}}{64}+\epsilon\bigg)\frac{(1-a^2)^2r^2}{(1-r^2)(1-a^4r^2)}\\&=1+\frac{(1-a)^3F_2(a)}{(4+(-3+\sqrt{17})a)^2(-80+48\sqrt{17}+4(109-27\sqrt{17})a^4)}\\&\quad\quad\quad+\epsilon\frac{64(1-a^2)^2\big(52-12\sqrt{17}+(-180+44\sqrt{17})a+(161-39\sqrt{17})a^2\big)}{(4+(-3+\sqrt{17})a)^2(-80+48\sqrt{17}+4(109-27\sqrt{17})a^4)}\\&=1+(1-a)^2\bigg(\frac{(1-a)F_2(a)}{(4+(-3+\sqrt{17})a)^2(-80+48\sqrt{17}+4(109-27\sqrt{17})a^4)}\\&\quad+\epsilon\frac{64(1+a)^2\big(52-12\sqrt{17}+(-180+44\sqrt{17})a+(161-39\sqrt{17})a^2\big)}{(4+(-3+\sqrt{17})a)^2(-80+48\sqrt{17}+4(109-27\sqrt{17})a^4)}\bigg),
	\end{align*}
	where
	\begin{align*}
		F_2(a)=&(27560-6808\sqrt{17})+(3120-848\sqrt{17})a+(-48382+11698\sqrt{17})a^2\\&+(7854-1922\sqrt{17})a^3+(47848-11608\sqrt{17})a^4+(25496-6184\sqrt{17})a^5.
	\end{align*}
	 From here, we see that in the case when $a\rightarrow 1^-$, the right-hand side of the expression $\mathcal{L}_{1,f^{*}_a}(r)$ behaves like $1+C_1(1-a)^2\epsilon$, which is greater than $1$ (i.e., $>1$),
	 where
	 \begin{align*}
	 	\quad\quad \quad C_1=\frac{64(1+a)^2\big(52-12\sqrt{17}+(-180+44\sqrt{17})a+(161-39\sqrt{17})a^2\big)}{(4+(-3+\sqrt{17})a)^2(-80+48\sqrt{17}+4(109-27\sqrt{17})a^4)}.
	 \end{align*}
	 This completes the proof.
\end{proof}	
\section{\bf Generalization of the Bohr inequality involving a sequence of differential functions for a certain classes of harmonic mappings}
Let $ \mathcal{H}(\Omega) $ be the class of complex-valued functions harmonic in a simply connected domain $ \Omega $. It is well-known that functions $ f $ in the class $ \mathcal{H}(\Omega) $ has the following representation $ f=h+\overline{g} $, where $ h $ and $ g $ both are analytic functions in $ \Omega $. The famous Lewy's theorem \cite{Lew-BAMS-1936} in $ 1936 $ states that a harmonic mapping $ f=h+\overline{g} $ is locally univalent on $ \Omega $ if, and only if, the determinant $ |J_f(z)| $ of its Jacobian matrix $ J_f(z) $ does not vanish on $ \Omega $, where
\begin{equation*}
	|J_f(z)|:=|f_{z}(z)|^2-|f_{\bar{z}}(z)|^2=|h^{\prime}(z)|^2-|g^{\prime}(z)|^2\neq 0.
\end{equation*}
In view of this result, a locally univalent harmonic mapping  is sense-preserving if $ |J_f(z)|>0 $ and sense-reversing if $|J_{f}(z)|<0$ in $\Omega$. For detailed information about the harmonic mappings, we refer the reader to \cite{Clunie-AASF-1984,Duren-2004}. In \cite{Kay & Pon & Sha & MN & 2018}, Kayumov \emph{et al.} first established the harmonic extension of the classical Bohr theorem, since then investigating on the Bohr-type inequalities for certain class of harmonic mappings becomes an interesting topic of research in geometric function theory.\vspace{1.2mm}

Methods of harmonic mappings have been applied to study and solve the fluid flow problems (see \cite{Aleman-2012,Constantin-2017}). In particular, univalent harmonic mappings having special geometric
properties such as convexity, starlikeness and close-to-convexity arise naturally in planar fluid dynamical problems. For example, in $2012$, Aleman and Constantin \cite{Aleman-2012} established a connection between harmonic mappings and ideal fluid flows. In fact, Aleman and Constantin have developed ingenious technique to solve the incompressible two dimensional Euler equations in terms of univalent harmonic mappings (see \cite{Constantin-2017} for details).\vspace{1.2mm}

Bohr phenomenon can be studied in view of the Euclidean distance and in this paper, we study the same for certain classes of harmonic mappings. The classical Bohr inequality can be written as
\begin{equation}\label{e-1.2}
	d\left(\sum_{n=0}^{\infty}|a_nz^n|,|a_0|\right)=\sum_{n=1}^{\infty}|a_nz^n|\leq 1-|f(0)|=d(f(0),\partial (\mathbb{D})),
\end{equation}
where $ d $ is the Euclidean distance.

In \cite{Abu-CVEE-2010}, Abu-Muhanna have obtained the following result for subordination class $ \mathcal{S}(f) $ when $ f $ is univalent.
\begin{thmG}\emph{(\cite[Theorem 1]{Abu-CVEE-2010})}
	If $ g=\sum_{n=0}^{\infty}b_nz^n\in\mathcal{S}(f) $, and $ f=\sum_{n=0}^{\infty}a_nz^n $ is univalent, then 
	\begin{align*}
		\sum_{n=1}^{\infty}|b_n|r^n\leq d(f(0), \partial\Omega),
	\end{align*}
	for $ |z|=\rho^*_0\leq 3-\sqrt{8}\approx0.17157 $, where $ \rho^*_0 $ is sharp for the Koebe function $ f(z)=z/(1-z)^2 $.
\end{thmG}

\par Let $ \mathcal{H} $ be the class of all complex-valued harmonic functions $ f=h+\bar{g} $ defined on the unit disk $ \mathbb{D} $, where $ h $ and $ g $ are analytic in $ \mathbb{D} $ with the normalization $ h(0)=h^{\prime}(0)-1=0 $ and $ g(0)=0 $. Let $ \mathcal{H}_0 $ be defined by 
\begin{align*}
	\mathcal{H}_0=\{f=h+\bar{g}\in\mathcal{H} : g^{\prime}(0)=0\}.
\end{align*} Therefore, each $f=h+\overline{g}\in \mathcal{H}_{0}$ has the following representation 
\begin{equation}\label{e-7.2}
	f(z)=h(z)+\overline{g(z)}=\sum_{n=1}^{\infty}a_nz^n+\overline{\sum_{n=1}^{\infty}b_nz^n}=z+\sum_{n=2}^{\infty}a_nz^n+\overline{\sum_{n=2}^{\infty}b_nz^n},
\end{equation}
where $ a_1 = 1 $ and $ b_1 = 0 $, since $ a_1 $ and $ b_1 $ have been appeared in later results and corresponding proofs.\vspace{2mm}

For analytic and geometric properties of harmonic functions, we refer to \cite{Duren-2004}. In $2013$, Ponnusamy \emph{et al.} (see \cite{Ponn-Yama-Yana-CVEE-2013}) considered the following classes 
\begin{align*}
	\mathcal{P}_{\mathcal{H}}=\{f=h+\overline{g} : {\rm Re} h'(z)>|g'(z)|\; \mbox{for}\; z\in\mathbb{D}\}
\end{align*}
and $ \mathcal{P}^0_{\mathcal{H}}=\mathcal{P}_{\mathcal{H}}\cap \mathcal{H}_0$.\vspace{2mm}

Let $\{\varphi_n(r)\}^{\infty}_{n=0}$ be a sequence of non-negative continuous function in $[0,1)$ such that the series $\sum_{n=0}^{\infty}\varphi_n(r)$ converges locally uniformly on the interval $[0,1)$. With the help of this idea of the change of basis from $\{r^n\}^{\infty}_{n=0}$ to  $\{\varphi_n(r)\}^{\infty}_{n=0}$, Kayumov \emph{et al.} (see \cite[Theorem 1]{Kayumov-MJM-2022}) generalized the classical Bohr inequality and established.
\begin{thmH}\emph{(\cite[Theorem 1]{Kayumov-MJM-2022})}
	Let  $ f\in\mathcal{B} $ with $f(z)=\sum_{n=0}^{\infty} a_{n}z^{n}$ and $p\in (0,2]$. If 
	\begin{align*}
		\varphi_0(r)>\frac{2}{p}\sum_{n=1}^{\infty}\varphi_n(r)\; \mbox{for}\; r \in [0,R),
	\end{align*}where $R$ is the minimal positive root of the equation $	\varphi_0(x)=\frac{2}{p}\sum_{n=1}^{\infty}\varphi_n(x),$ then the following sharp inequality holds:
	\begin{align*}
		B_f(\varphi,p,r):=|a_0|^p\varphi_0(r)+\sum_{n=1}^{\infty}|a_n|\varphi_n(r)\leq \varphi_0(r)\;\mbox{for all}\; r\leq R.
	\end{align*}In the case when $\varphi_0(r)<\frac{2}{p}\sum_{n=1}^{\infty}\varphi_n(r)$ in some interval $(R,R+\epsilon)$, the number $R$ cannot be improved.If the functions $\varphi_n(x)\;( n\geq0)$ are smooth functions then the last condition is equivalent to the inequality $	\varphi^{\prime}_0(R)<\frac{2}{p}\sum_{n=1}^{\infty}\varphi^{\prime}_n(R).$
\end{thmH}
The value of the Bohr radius for the identity operator and the convolution operator with the hypergeometric Gaussian function was found by the authors using Theorem H. Let 
\begin{align*}
	F(z)=F(a,b,c,z)=\sum_{n=0}^{\infty}\gamma_nz^n,\; \gamma_n:=\frac{(a)_n(b)_n}{(c)_n(1)_n},
\end{align*}
where $a, b, c>-1$, such that $\gamma_n\geq 0$, $(a)_n:=a(a+1)\cdots(a+n-1)$, $(a)_0=1$. Then  for $p\in (0, 2]$
\begin{align*}
	|a_0|^p+\sum_{n=1}^{\infty}|\gamma_n||a_n|r^n\leq \; \mbox{for all}\; r\leq R,
\end{align*}
where $R$ is the minimal positive root of the equation $|F(a,b,c,x)-1|=p/2$. Recently, Chen \emph{et al.} \cite{Chen-Liu-Pon-RM-2023} established several new versions of Bohr-type inequalities for bounded analytic functions in the unit disk $\mathbb{D}$ by allowing $\{\varphi_n(r)\}^{\infty}_{n=0}$ in the place of $\{r^n\}^{\infty}_{n=0}$ in the power series representations of the functions involved with the Bohr sum and thereby introducing a single parameter, which generalized several related results.

Let $\mathcal{G}$ denote the set of all sequences $\{\varphi_n(r)\}^{\infty}_n$ of non-negative differentiable increasing functions in $[0,1)$ such that the series $\sum_{n=0}^{\infty}\varphi_n(r)$ and $\sum_{n=0}^{\infty}\varphi_n^{\prime}(r)$ converges locally uniformly on the interval $[0,1)$,$\;\forall\; n\in\mathbb{N}\cup \{0\}$. 

We define the Bohr phenomenon, generalized with respect to a sequence of functions, for a certain class of harmonic mappings.
\begin{defi}
	The Bohr phenomenon for certain harmonic class $\mathcal{F}$ consisting of functions $f$ of the form \eqref{e-7.2} involving a sequence of functions $\{\varphi_n(r)\}^{\infty}_{n=0}\in\mathcal{G}$ is as follows: find the largest radius $r_f\in (0, 1)$ such that the following inequality
	\begin{align}\label{Eq-33.33}
		A_f(r)=|z|\varphi_0(r)+\sum_{n=2}^{\infty}\left(|a_n|+|b_n|\right)\varphi_n(r)\leq d\left(f(0), \partial f(\mathbb{D})\right),
	\end{align}
	holds for $|z|=r\leq r_f$ and for all $f\in\mathcal{F}$, where $\mathcal{A}_f(r)$ is generalized Majorant series of $f$. The radius $r_f$ is called the Bohr radius for the class $\mathcal{F}$. This radius is considered best possible if a function $g \in \mathcal{F}$ exists such that for any $r > r_f$, the inequality $\mathcal{A}_g(r) > d\left(f(0), \partial f(\mathbb{D})\right)$ holds.
\end{defi}

\subsection{\bf Improved Bohr inequalities for certain class of harmonic univalent functions}

Inspired by the idea of the proof technique of   \cite{Kayumov-MJM-2022,Chen-Liu-Pon-RM-2023}, our aim is to establish refined Bohr inequality, finding the corresponding sharp radius for the class $ \mathcal{P}^{0}_{\mathcal{H}}(M) $ which has been studied by Ghosh and Vasudevarao (see \cite{Ghosh-Vasudevarao-BAMS-2020}) 
\begin{align*}
	\mathcal{P}^{0}_{\mathcal{H}}(M)=\{f=h+\overline{g} \in \mathcal{H}_{0}: \real (zh^{\prime\prime}(z))> -M+|zg^{\prime\prime}(z)|, \; z \in \mathbb{D}\; \mbox{and }\; M>0\}.
\end{align*}
To investigate the Bohr inequality and Bohr radius for functions in $\mathcal{P}^{0}_{\mathcal{H}}(M)$, we first require the coefficient bounds and growth estimate for this class of functions. The following result provides these necessary estimates.
\begin{lemD}(\cite[Theorem 2.3]{Ghosh-Vasudevarao-BAMS-2020})
	Let $f=h+\overline{g}\in \mathcal{P}^{0}_{\mathcal{H}}(M)$ be given by \eqref{e-7.2} for some $M>0$. Then for $n\geq 2,$ 
	\begin{enumerate}
		\item[(i)] $\displaystyle |a_n| + |b_n|\leq \frac {2M}{n(n-1)}; $\\[2mm]
		
		\item[(ii)] $\displaystyle ||a_n| - |b_n||\leq \frac {2M}{n(n-1)};$\\[2mm]
		
		\item[(iii)] $\displaystyle |a_n|\leq \frac {2M}{n(n-1)}.$
	\end{enumerate}
	The inequalities  are sharp with extremal function   $f_M$ given by 
	$f_M^{\prime}(z)=1-2M\, \ln\, (1-z) .$	
\end{lemD}
\begin{lemE}(\cite[Theorem 2.4]{Ghosh-Vasudevarao-BAMS-2020})
	Let $f \in \mathcal{P}^{0}_{\mathcal{H}}(M)$ be given by \eqref{e-7.2}. Then 
	\begin{equation*}
		L_M(r)=r +2M \sum\limits_{n=2}^{\infty} \dfrac{(-1)^{n-1}r^{n}}{n(n-1)} \leq |f(z)| \leq r + 2M \sum\limits_{n=2}^{\infty} \dfrac{r^{n}}{n(n-1)}=R_M(r)
	\end{equation*}
	for $|z|=r\in (0, 1)$. Both  inequalities are sharp for the function $f_{M}$ given by 
	\begin{align*}
		f_{M}(z)=z+ 2M \sum\limits_{n=2}^{\infty} \dfrac{z^n}{n(n-1)}.
	\end{align*}
\end{lemE}
For the class $\mathcal{P}^{0}_{\mathcal{H}}(M)$, Allu and Halder (see \cite{Allu-BSM-2021}) studied the Bohr inequality in terms of distance formulations and obtained the result.
\begin{thmI}\emph{(\cite[Theorem 2.9]{Allu-BSM-2021})}
	Let $f\in \mathcal{P}^{0}_{\mathcal{H}}(M) $ be given by \eqref{e-7.2} with  $0<M<1/(2(\ln 4-1))$ . Then
	\begin{align}\label{Eq-33.44}
		|z|+\sum_{n=2}^{\infty}\left(|a_n|+|b_n|\right)|z|^n\leq d\left(f(0), \partial f(\mathbb{D})\right)
	\end{align} holds for $|z|=r\leq r_f$, where $r_f$ is the unique root of 
	$R_M(r)=L_M(1)$ in $(0,1)$. The radius $r_f$ is the best possible.
\end{thmI}
It is natural to raise the following question for further exploration of Bohr inequality for the class  $\mathcal{P}^{0}_{\mathcal{H}}(M)$.
\begin{ques}\label{Qn-3.1}
	Can we establish a generalization of Theorem I involving a sequence $\{\varphi_n(r)\}_{n=1}^{\infty}\in\mathcal{G}$?
\end{ques}
We begin by presenting our result concerning the Bohr inequality for the class $\mathcal{P}^{0}_{\mathcal{H}}(M)$ in distance formulation, addressing Question \ref{Qn-3.1}. Notably, if we set $\varphi_n(r)=r^{n}$ (for $n\in\mathbb{N}\cup\{0\}$) in Theorem \ref{th-7.3}, then Theorem I becomes a special case of Theorem \ref{th-7.3}.
\begin{thm}\label{th-7.3}
	Let $f=h+\overline{g}\in \mathcal{P}^{0}_{\mathcal{H}}(M)$ be given by \eqref{e-7.2} with 
	\begin{align}\label{ee-7.6}
		0<M<\frac{1}{2\big(\sum_{n=2}^{\infty}\frac{\phi_n(0)}{n(n-1)}-1+\ln 4\big)}
	\end{align} and suppose that $\{\varphi_n(r)\}^{\infty}_{n=0}\in\mathcal{G}$ with condition $\varphi_n(1)=1$. Then \eqref{Eq-33.33} holds for $|z|=r\leq R_f(M),$ where $R_f(M)$ is the unique positive root in $(0,1)$ of the equation
	\begin{align}\label{Eq-33.66}
		r\varphi_0(r)+2M\sum_{n=2}^{\infty}\frac{\varphi_n(r)}{n(n-1)}=L_M(1).
	\end{align}The radius $R_f(M)$ is best possible.
\end{thm}
\begin{rem}
	Setting $\varphi_n(r)=r^{n}$ (for $n\in\mathbb{N}\cup\{0\}$) in Theorem \ref{th-7.3} (noting that $\{r^n\}_{n=0}^{\infty}$ is a sequence of differentiable, increasing functions on $[0, 1)$ and $\{r^n\}_{n=0}^{\infty}\in\mathcal{G}$), the inequality \eqref{Eq-33.66} simplifies to $R_M(r)=L_M(1)$. This confirms that $R_f(M)=r_f$ and establishes Theorem I as a special case of Theorem \ref{th-7.3}. Moreover, we have
	\begin{align*}
		\sum_{n=0}^{\infty}\varphi_n(r)=\frac{1}{1-r}\,\,\mbox{for}\;\;r\in(0,1).
	\end{align*}
	The series converges pointwise on $[0,1)$. Since the limit function, $g(r)=\frac{1}{1-r}$, is continuous on any closed sub-interval $[0,a]$ for $a<1$, and the convergence is uniform on every compact subset of $[0,1)$, the series converges locally uniformly on $[0,1)$.\vspace{2mm}
	
	\noindent Again, we see that 
	\begin{align*}
		\sum_{n=0}^{\infty}\varphi_n^{\prime}(r)=	\sum_{n=0}^{\infty}nr^{n-1}\,\,\mbox{for}\;\;r\in(0,1)
	\end{align*}
	which is in fact
	\begin{align*}
		\frac{d}{dr}\bigg(\sum_{n=0}^{\infty}\varphi_n(r)\bigg)=\sum_{n=1}^{\infty}nr^{n-1}=\frac{1}{(1-r)^2}\;\;\mbox{for}\;\; r\in(0,1).
	\end{align*}
	The series also converges pointwise on $[0,1)$. Since the limit function, $g_1(r) = \frac{1}{(1-r)^2}$, is continuous on any closed sub-interval $[0,a]$ for $a<1$, and the convergence is uniform on every compact subset of $[0,1)$, the series converges locally uniformly on $[0,1)$.
\end{rem}
\subsection{\bf Some consequences of Theorem \ref{th-7.3}}
Let $f=h+\overline{g}\in \mathcal{P}^{0}_{\mathcal{H}}(M)$ be given by \eqref{e-7.2} for  $0<M<1/2(\ln 4-1)$. A simple computation shows that
\begin{align*}
	\begin{cases}
		\displaystyle\sum_{n=2}^{\infty}\frac{nr^n}{n(n-1)}=-r\ln(1-r),\vspace{2mm}\\	\displaystyle\sum_{n=2}^{\infty}\frac{n^2r^n}{n(n-1)}=\frac{r[r-(1-r)\ln(1-r)]}{1-r},\vspace{2mm}\\
		\displaystyle\sum_{n=2}^{\infty}\frac{n^3r^n}{n(n-1)}=r\left(\frac{(3-2r)r}{(1-r)^2}-\ln (1-r)\right).
	\end{cases}
\end{align*}

\noindent{\bf (I):} For $ f=h+\overline{g}\in \mathcal{P}^{0}_{\mathcal{H}}(M)$ and for $\varphi_0(r)=1$ and $\varphi_n(r)=n^{\alpha}r^n\;(n\geq1)$, we define the following functional
	\begin{align*}
		\mathcal{C}^f_{\alpha}(r):=r\varphi_0(r)+\sum_{n=2}^{\infty}\left(|a_n|+|b_n|\right)\varphi_n(r).
	\end{align*} 
	We see that $\varphi_0(1)=1$ and the condition
	\begin{align*}
		0<M<\frac{1}{2\big(\sum_{n=2}^{\infty}\frac{\phi_n(0)}{n(n-1)}-1+\ln 4\big)}\; \mbox{coincides with}\; 0<M<\frac{1}{2(\ln 4-1)}.
	\end{align*}
	Then, we have\vspace{2mm}

\noindent{\bf (A):} $\mathcal{C}^f_{1}(r)\leq d\left(f(0), \partial f(\mathbb{D})\right)$ for $|z|=r\leq R_1(M)$, where $R_1(M)$ is the unique positive root in $(0,1)$ of the equation $r-2Mr\ln(1-r)=1+2M(1-\ln 4)$. The number $R_1(M)$ is best possible.\vspace{2mm}
		
\noindent{\bf (B):} $\mathcal{C}^f_{2}(r)\leq d\left(f(0), \partial f(\mathbb{D})\right)$ 
		for $|z|=r\leq R_2(M)$, where $R_2(M)$ is the unique positive root in $(0,1)$ of the equation 
		\begin{align*}
			r+\frac{2Mr\left(r-(1-r)\ln(1-r)\right)}{1-r}=1+2M(1-\ln 4).
		\end{align*}
		The number $R_2(M)$ is best possible.\vspace{2mm}
		
		\noindent{\bf (C):} $\mathcal{C}^f_{3}(r)\leq d\left(f(0), \partial f(\mathbb{D})\right)$ for $|z|=r\leq R_3(M)$, where $R_3(M)$ is the unique positive root in $(0,1)$ of the equation 
		\begin{align*}
			r+2Mr\left(\frac{(3-2r)r}{(1-r)^2}-\ln (1-r)\right)=1+2M(1-\ln 4).
		\end{align*}
		The number $R_3(M)$ is best possible.
	\begin{table}[ht]
		\centering
		\begin{tabular}{|l|l|l|l|l|l|l|l|l|l|}
			\hline
			$M$& $0.431 $&$0.862 $& $1.210 $& $1.271$& $1.289 $&$1.292 $&$ 1.2935 $ &$1.29421$& $1.29433$ \\
			\hline
			$R_{1}(M)$& $0.443$&$0.230 $& $0.057$& $0.017$& $0.0040 $&$0.0018 $& $0.00065$& $0.00010$ &$ 0.000015 $\\
			\hline
			$R_{2}(M)$& $0.358$&$0.189 $& $0.029$& $0.016$& $0.0040 $&$0.0017 $& $0.00065$& $0.00010$ &$ 0.000015 $\\
			\hline
			$R_{3}(M)$& $0.277$&$0.149 $& $0.044$& $0.015$& $0.0039 $&$0.0017 $& $0.00065$& $0.00010$ &$ 0.000015 $\\
			\hline
		\end{tabular}\vspace{2.5mm}
		\caption{This table exhibits the approximate values of the roots $R_{\alpha}(M)$ for different values of $M$ and $\alpha=1, 2, 3$, where $0 < M < 1/(2(\ln 4 - 1)) \approx 1.29435$. Moreover, we see that $\lim_{M\rightarrow 1.29435^-} R_{\alpha}(M) \approx 0.000015$.}
	\end{table}
\noindent{\bf (II):}  A simple computation shows that
	\begin{align*}
		\begin{cases}
			\displaystyle\sum_{n=2}^{\infty}\frac{(n+1)r^n}{n(n-1)}=r+(1-2r)\ln(1-r),\vspace{2mm}\\	\displaystyle\sum_{n=2}^{\infty}\frac{(n+1)^2r^n}{n(n-1)}=\frac{r+(1-5r+4r^2)\ln(1-r)}{1-r},\vspace{2mm}\\
			\displaystyle\sum_{n=2}^{\infty}\frac{(n+1)^3r^n}{n(n-1)}=\frac{r+4r^2-4r^3+(1-r)^2(1-8r)\ln (1-r)}{(1-r)^2}.
		\end{cases}
	\end{align*} 
	For $ f=h+\overline{g}\in \mathcal{P}^{0}_{\mathcal{H}}(M)$ and for $\varphi_n(r)=(n+1)^{\beta}r^n\;(n\geq1)$, we define the following functional
	\begin{align*}
		\mathcal{D}^f_{\beta}(r):=r\varphi_0(r)+\sum_{n=2}^{\infty}\left(|a_n|+|b_n|\right)\varphi_n(r).
	\end{align*}
	We see that $\varphi_0(1)=1$ and the condition
	\begin{align*}
		0<M<\frac{1}{2\big(\sum_{n=2}^{\infty}\frac{\phi_n(0)}{n(n-1)}-1+\ln 4\big)}\; \mbox{coincides with}\; 0<M<\frac{1}{2(\ln 4-1)}.
	\end{align*}
	Then, we have\vspace{2mm}

\noindent{\bf (A):} $\mathcal{D}^f_{1}(r) \leq d\left(f(0), \partial f(\mathbb{D})\right)$ for $|z|=r\leq R^*_1(M)$, where $R^*_1(M)$ is the unique positive root in $(0,1)$ of the equation $r+2M[r+(1-2r)\ln(1-r)]=1+2M(1-\ln 4).$ The number $R_1^*(M)$ is best possible.\vspace{2mm}
		
\noindent{\bf (B):} $\mathcal{D}^f_{2}(r) \leq d\left(f(0), \partial f(\mathbb{D})\right)$ for $|z|=r\leq R^*_2(M)$, where $R_2(M)$ is the unique positive root in $(0,1)$ of the equation 
		\begin{align*}
			r+\frac{2M[r+(1-5r+4r^2)\ln(1-r)]}{1-r}=1+2M(1-\ln 4).
		\end{align*}
		The number $R^*_2(M)$ is best possible.\vspace{2mm}
		
\noindent{\bf (C):} $\mathcal{D}^f_{3}(r) \leq d\left(f(0), \partial f(\mathbb{D})\right)$
		for $|z|=r\leq R^*_3(M)$, where $R^*_3(M)$ is the unique positive root in $(0,1)$ of the equation 
		\begin{align*}
			r+\frac{2M[r+4r^2-4r^3+(1-r)^2(1-8r)\ln (1-r)]}{(1-r)^2}=1+2M(1-\ln 4).
		\end{align*}
		The number $R^*_3(M)$ is best possible.
\begin{table}[ht]
	\centering
	\begin{tabular}{|l|l|l|l|l|l|l|l|l|l|}
		\hline
		$M$& $0.431 $&$0.862 $& $1.210 $& $1.271$& $1.289 $&$1.292 $&$ 1.2935 $ &$1.29421$& $1.29433$ \\
		\hline
		$R^*_{1}(M)$& $0.404$&$0.208 $& $0.054$& $0.016$& $0.0040 $&$0.0018 $& $0.00065$& $0.00010$ &$ 0.000015 $\\
		\hline
		$R^*_{2}(M)$& $0.284$&$0.147 $& $0.043$& $0.015$& $0.0039 $&$0.0017 $& $0.00065$& $0.00010$ &$ 0.000015 $\\
		\hline
		$R^*_{3}(M)$& $0.203$&$0.147 $& $0.043$&  $0.015$& $0.0039 $&$0.0017 $& $0.00065$& $0.00010$ &$ 0.000015 $\\
		\hline
	\end{tabular}\vspace{2.5mm}
	\caption{This table exhibits that the approximate values of the roots $ R^*_{\beta}(M) $ for different values of $ M$ and  $\beta=1,2,3$, where $0<M<1/(2(\ln 4-1))\approx 1.29435$. Moreover, we see that $ \lim\limits_{M\rightarrow 1.29435}R^*_{\beta}(M)\approx 0.000015$.}
\end{table}
\subsection{\bf Proof of Theorem \ref{th-7.3}} 
Let $f = h + \overline{g} \in \mathcal{P}^{0}_{\mathcal{H}}(M)$ be given by equation \eqref{e-7.2} for $0 < M < 1/(2(\ln 4 - 1))$. Then, in view of Lemma E, we have
\begin{align}\label{e-7.6}
	|f(z)|\geq|z|+2M\sum_{n=2}^{\infty}\frac{(-1)^{n-1}|z|^n}{n(n-1)},\;\mbox{for}\;|z|<1.
\end{align}
Then it follows from \eqref{e-7.6} that
\begin{align}\label{Eq-3.5}
	\liminf_{|z|\rightarrow1} |f(z)|\geq1+2M\sum_{n=2}^{\infty}\frac{(-1)^{n-1}}{n(n-1)}	.
\end{align}
The Euclidean distance between $f(0)$ and the boundary of $f(\mathbb{D}$ is given by 
\begin{align}\label{Eq-3.6}
	d(f(0),\partial f(\mathbb{D}))=\liminf_{|z|\rightarrow1}|f(z)-f(0)|.
\end{align}
Since $f(0)=0$, it follows that $|f(z)-f(0)| = |f(z)|$. Therefore, using \eqref{Eq-3.5} and \eqref{Eq-3.6}, we obtain
\begin{align}\label{Eq-3.7}
	d(f(0),\partial f(\mathbb{D}))\geq 1+2M\sum_{n=2}^{\infty}\frac{(-1)^{n-1}}{n(n-1)}.	
\end{align}
Let $H_M : [0, 1]\to\mathbb{R}$ be defined by 
\begin{align}\label{e-7.7}
	H_M(r)=r\varphi_0(r)+2M\sum_{n=2}^{\infty}\frac{\varphi_n(r)}{n(n-1)}-1-2M\sum_{n=2}^{\infty}\frac{(-1)^{n-1}}{n(n-1)}.
\end{align} 
Clearly, $H_M$ is continuous in $[0,1]$ and $H_M$ is differentiable $(0,1)$ for the following reason.

\noindent{\bf (R1):}  Since $\varphi_0(r)$ is differentiable on $(0,1)$, the expression $r\varphi_0(r) - 1 - 2M\sum_{n=2}^{\infty} \frac{(-1)^{n-1}}{n(n-1)}$ is also differentiable on $(0,1)$, as it is equal to $r\varphi_0(r) - 1 - 2M(1 - \ln 4)$. \vspace{2mm}
	
\noindent{\bf (R2):} Since the series $\sum_{n=0}^{\infty} \varphi_n(r)$ is locally uniformly convergent on $[0, 1)$, it is uniformly convergent on every closed interval $[0, a] \subset [0, 1)$. Therefore, there exists $K > 0$ such that for all $n$ and $r \in [0,a]$, $|\varphi_n(r)| \le K$.

 We define
	\begin{align*}
		S(r)=\sum_{n=2}^{\infty}S_n(r),\;\;\mbox{where}\;\;S_n(r)=\frac{\varphi_n(r)}{n(n-1)}.
	\end{align*}
	Then, for $r\in[0,a]$, we have	
	\begin{align*}
		|S_n(r)|=\bigg|\frac{\varphi_n(r)}{n(n-1)}\bigg|\leq\frac{K}{n(n-1)}:=K_n.
	\end{align*}
	We see that $\sum_{n=2}^{\infty}K_n = K \sum_{n=2}^{\infty} \frac{1}{n(n-1)} = K$. The series $\sum_{n=2}^{\infty}K_n$ is clearly convergent. Hence, by the Weierstrass $M$-test, the series $\sum_{n=2}^{\infty}S_n(r)$ converges uniformly and absolutely on $[0, a]$, and therefore converges locally uniformly on $[0,1)$.\vspace{2mm}
	
	Since the series $\sum_{n=0}^{\infty} \varphi^{\prime}_n(r)$ is locally uniformly convergent on $[0, 1)$, for every closed interval $[0, a] \subset [0, 1)$, there exists $K_1 > 0$ such that for all $n$ and $r \in [0,a]$, $|\varphi^{\prime}_n(r)|\leq K_1.$ 
	Then, for $r\in[0,a]$, we have
	\begin{align*}
		|S_n^{\prime}(r)|=\bigg|\frac{\varphi^{\prime}_n(r)}{n(n-1)}\bigg|\leq\frac{K_1}{n(n-1)}:=K^*_n.
	\end{align*}
	Again, we see that $\sum_{n=2}^{\infty}K^*_n = K_1 \sum_{n=2}^{\infty}\frac{1}{n(n-1)} = K_1$. The series $\sum_{n=2}^{\infty}K^*_n$ is clearly convergent. Therefore, by the Weierstrass $M$-test, the series $\sum_{n=2}^{\infty}S^{\prime}_n(r)$ converges uniformly and absolutely on $[0,a]$, and consequently, it converges locally uniformly on $[0, 1)$.
	\vspace{2mm}
	
Since each $S_n(r)$ is differentiable on $(0,1)$, and both $\sum_{n=2}^{\infty} S_n(r)$ and $\sum_{n=2}^{\infty} S'_n(r)$ converge locally uniformly on $[0, 1)$, it follows from the theorem on term-by-term differentiation (see \cite[Theorem 7.17]{Rudin-1976}) that the sum, $\sum_{n=2}^{\infty} S_n(r)$, is differentiable on $(0,1)$, and its derivative is the sum of the derivatives, $\sum_{n=2}^{\infty} S'_n(r)$. Thus, we have 
	\begin{align*}
		\frac{d}{dr}\big(S(r)\big)=\frac{d}{dr}\bigg(\sum_{n=2}^{\infty}S_n(r)\bigg)=\sum_{n=2}^{\infty}\frac{d}{dr}\big(S_n(r)\big)
	\end{align*}
	which is equivalent to 
	\begin{align*}
		\frac{d}{dr}\bigg(\sum_{n=2}^{\infty}\frac{\varphi_n(r)}{n(n-1)}\bigg)=\sum_{n=2}^{\infty}\frac{\varphi_n^{\prime}(r)}{n(n-1)}.
	\end{align*}
	Consequently, we have
	\begin{align}\label{Eq-33.99}
		\frac{d}{dr}\bigg(2M\sum_{n=2}^{\infty}\frac{\varphi_n(r)}{n(n-1)}\bigg)=2M\sum_{n=2}^{\infty}\frac{\varphi^{\prime}_n(r)}{n(n-1)}.
	\end{align}

Since each term in \eqref{e-7.7} is differentiable, their finite sum $H_M(r)$ is also differentiable on $(0,1)$. Moreover, in view of the condition in \eqref{ee-7.6}, we see that
\begin{align*}
	H_M(0)=2M\sum_{n=2}^{\infty}\frac{\varphi_n(0)}{n(n-1)}-1-2M\sum_{n=2}^{\infty}\frac{(-1)^{n-1}}{n(n-1)}<0.
\end{align*}
Since $\varphi_n(1)=1$ for all $n\in\mathbb{N}\cup\{0\}$ and $1-\ln 4\approx -0.3862<0$, we have
\begin{align*}
	H_M(1)=&\varphi_0(1)+2M\sum_{n=2}^{\infty}\frac{\varphi_n(1)}{n(n-1)}-1-2M\sum_{n=2}^{\infty}\frac{(-1)^{n-1}}{n(n-1)}\\&=2M\left(\sum_{n=2}^{\infty}\frac{\varphi_n(1)}{n(n-1)} -(1-\ln 4)\right)>0.
\end{align*}
By the Intermediate Value Theorem, the function $H_M$ has a root in $(0,1)$, which we denote by $R_f(M)$. To show that $R_f(M)$ is unique, it is sufficient to show that $H_M$ is a monotonic function on $[0, 1]$. Since the set of all sequences $\{\varphi_n(r)\}^{\infty}_{n=0}$ of non-negative, differentiable functions on $[0,1)$ such that the series $\sum_{n=0}^{\infty}\varphi_n(r)$ converges locally uniformly on the interval $[0,1)$, we proceed as follows: an easy computation using \eqref{Eq-33.99} shows that
\begin{align*}
	H_M^{\prime}(r)=r\varphi^{\prime}_0(r)+\varphi_0(r)+2M\sum_{n=2}^{\infty}\frac{\varphi^{\prime}_n(r)}{n(n-1)}>0\; \mbox{for all}\; r\in (0,1).
\end{align*} 
Thus, we conclude that $H_M$ is a strictly monotonically increasing function on $(0,1)$. However, $H_M\left(R_f(M)\right)=0$ leads to
\begin{align}\label{e-7.8}
	R_f(M)\varphi_0(R_f(M))+2M\sum_{n=2}^{\infty}\frac{\varphi_n(R_f(M))}{n(n-1)}=1+2M\sum_{n=2}^{\infty}\frac{(-1)^{n-1}}{n(n-1)}.
\end{align} 
In view of \eqref{Eq-3.7}, Lemma D yields that 
\begin{align*}
	A_f(r)&=r\varphi_0(r)+\sum_{n=2}^{\infty}\left(|a_n|+|b_n|\right)\varphi_n(r)\\&\leq r\varphi_0(r)+2M\sum_{n=2}^{\infty}\frac{\varphi_n(r)}{n(n-1)}\\&\leq 1+2M\sum_{n=2}^{\infty}\frac{(-1)^{n-1}}{n(n-1)}\\&\leq d\left(f(0), \partial f(\mathbb{D})\right) 
\end{align*} 
for $|z|=r\leq R_f(M)$. Thus the desired inequality is established.\vspace{1.2mm}

The next part of the proof requires showing that $R_f(M)$ is best possible. Therefore, henceforth, we consider the function $f_M$ given by
\begin{align*}
	f_M(z)=z+2M\sum_{n=2}^{\infty}\frac{z^n}{n(n-1)}.
\end{align*}
Clearly, we see that $f_M \in \mathcal{P}^{0}_{\mathcal{H}}(M)$. For $r > R_f(M)$ and $f = f_M$, an easy computation in view of part (i) of Lemma D and \eqref{e-7.8} shows that
\begin{align*}
	A_{f_M}(r)&=r\varphi_0(r)+\sum_{n=2}^{\infty}\left(|a_n|+|b_n|\right)\varphi_n(r)\\&=r\varphi_0(r)+2M\sum_{n=2}^{\infty}\frac{\varphi_n(r)}{n(n-1)}\\&>R_f(M)\varphi_0(R_f(M))+2M\sum_{n=2}^{\infty}\frac{\varphi_n(R_f(M))}{n(n-1)}\\&=1+2M\sum_{n=2}^{\infty}\frac{(-1)^{n-1}}{n(n-1)}\\&=d\left(f_M(0), \partial f_M(\mathbb{D})\right).
\end{align*}
This turns out that the number $R_f(M)$ is best possible.
\subsection{\bf Improved Bohr inequality for a class of harmonic close-to-convex mappings}

In $1977$, Chichra \cite{Chichra-PAMS-1977} introduced the class $\mathcal{W}(\alpha)$ consisting of normalized analytic functions $h$ satisfying the condition $\text{Re}\left(h^{\prime}(z)+\alpha z h^{\prime\prime}(z)\right)>0$ for $z\in\mathbb{D}$ and $\alpha\geq 0$. Moreover, Chichra \cite{Chichra-PAMS-1977} showed that functions in the class $\mathcal{W}(\alpha)$ constitute a subclass of close-to-convex functions in $\mathbb{D}$. In $2014$, Nagpal and Ravichandran \cite{Nagpal-Ravinchandran-2014-JKMS} studied the following class 
\begin{equation*}
	\mathcal{W}^{0}_{\mathcal{H}}=\{f=h+\bar{g}\in \mathcal{H} :  {\rm Re}\left(h^{\prime}(z)+zh^{\prime\prime}(z)\right) > |g^{\prime}(z)+zg^{\prime\prime}(z)|\;\; \mbox{for}\; z\in\mathbb{D}\}
\end{equation*}
and obtained the coefficient bounds for the functions in the class $ \mathcal{W}^{0}_{\mathcal{H}} $. In 2019, Ghosh and Vasudevarao \cite{Nirupam-MonatsMath-2019} studied the class $\mathcal{W}^{0}_{\mathcal{H}}(\alpha)$, where 
\begin{equation*}
	\mathcal{W}^{0}_{\mathcal{H}}(\alpha)=\{f=h+\bar{g}\in \mathcal{H} :  {\rm Re}\left(h^{\prime}(z)+\alpha zh^{\prime\prime}(z)\right) > |g^{\prime}(z)+\alpha zg^{\prime\prime}(z)|\;\; \mbox{for}\; z\in\mathbb{D}\}.
\end{equation*}
From the following result, it is easy to see that functions in the class $\mathcal{W}^{0}_{\mathcal{H}}(\alpha)$ are univalent for $\alpha\geq 0$, and they are closely related to functions in $\mathcal{W}(\alpha)$.
\begin{lemF}\cite{Nirupam-MonatsMath-2019}
	The harmonic mapping $ f=h+\bar{g} $ belongs to $ \mathcal{W}^{0}_{\mathcal{H}}(\alpha) $ if, and only if, the analytic function $ F=h+\epsilon g $ belongs to $ \mathcal{W}(\alpha) $ for each $ |\epsilon|=1. $ 
\end{lemF}
The coefficient bounds and the sharp growth estimates for functions in the class $ \mathcal{W}^{0}_{\mathcal{H}}(\alpha) $ have been studied in \cite{Nirupam-MonatsMath-2019}.
\begin{lemG}\cite{Nirupam-MonatsMath-2019}
	Let $ f\in \mathcal{W}^{0}_{\mathcal{H}}(\alpha) $ for $ \alpha\geq 0 $ and be of the form \eqref{e-7.2}. Then for any $ n\geq 2 $,
	\begin{enumerate}
		\item[(i)] $ |a_n|+|b_n|\leq \displaystyle\frac{2}{\alpha n^2+(1-\alpha)n} $;\vspace{1.5mm}
		\item[(ii)] $  ||a_n|-|b_n||\leq \displaystyle\frac{2}{\alpha n^2+(1-\alpha)n} $; \vspace{1.5mm}
		\item[(iii)] $ |a_n|\leq \displaystyle\frac{2}{\alpha n^2+(1-\alpha)n} $.
	\end{enumerate}
	All these inequalities are sharp for the function $ f=f^*_{\alpha} $ given by 
	\begin{equation}\label{e-2.24}
		f^*_{\alpha}(z)=z+\sum_{n=2}^{\infty}\frac{2z^n}{\alpha n^2+(1-\alpha)n}.
	\end{equation}
\end{lemG}
\begin{lemH}\cite{Nirupam-MonatsMath-2019}
	Let $ f\in \mathcal{W}^{0}_{\mathcal{H}}(\alpha) $ and be of the form \eqref{e-7.2} with $ 0\leq\alpha<1 $. Then
	\begin{equation}\label{e-2.25}
		L_w(r)=r+\sum_{n=2}^{\infty}\frac{2(-1)^{n-1}r^n}{\alpha n^2+(1-\alpha)n}\leq |f(z)|\leq r+\sum_{n=2}^{\infty}\frac{2r^n}{\alpha n^2+(1-\alpha)n}=R_w(r)
	\end{equation}
	for $|z|=r\in(0, 1)$. Both the inequalities are sharp for the function  $ f=f^*_{\alpha} $ given by \eqref{e-2.24}.
\end{lemH}
Using Lemmas G and H, Allu and Halder (see \cite[Theorem 2.4]{Allu-BSM-2021}) obtained the Bohr radius for the class $\mathcal{W}^{0}_{\mathcal{H}}(\alpha)$.
\begin{thmJ}
	Let $f=h+\overline{g}\in \mathcal{W}^{0}_{\mathcal{H}}(\alpha)$ be given by \eqref{e-7.2}. Then the inequality \eqref{Eq-33.44} holds for $|z|=r\leq r_f(\alpha)$, where  $r_f(\alpha)$ is the unique root of $R_w(r)=L_w(1)$ in $(0, 1)$. The radius $r_f(\alpha)$ is the best possible.
\end{thmJ}
We establish the generalized Bohr inequality for the class $\mathcal{W}^{0}_{\mathcal{H}}(\alpha)$, incorporating a sequence of differentiable functions $\{\varphi_n(r)\}_{n=0}^{\infty}\in\mathcal{G}$.  The Hurwitz–Lerch Zeta function is utilized to explore subsequent results, where we show that the Bohr radius is best possible.
\begin{thm}\label{Th-4.1}
Let $f=h+\overline{g}\in \mathcal{W}^{0}_{\mathcal{H}}(\alpha)$ be given by \eqref{e-7.2}  and suppose that $\{\varphi_n(r)\}^{\infty}_{n=0}\in\mathcal{G}$ with condition $\varphi_n(1)=1$. Then \eqref{Eq-33.33} holds for $|z|=r\leq R_f(\alpha),$ where $R_f(\alpha)$ is the unique positive root in $(0,1)$ of the equation
\begin{align}\label{Eq-33.16}
	r\varphi_0(r)+2\sum_{n=2}^{\infty}\frac{\varphi_n(r)}{\alpha n^2+(1-\alpha)n}=L_w(1).
\end{align}The radius $R_f(\alpha)$ is best possible.
\end{thm}
\begin{rem} When $\varphi_n(r)=r^n$ (for $n\in\mathbb{N}\cup\{0\}$) is used in Theorem \ref{Th-4.1} , \eqref{Eq-33.16} simplifies to $R_w(r)=L_w(1)$ and $R_f(\alpha)=r_f(\alpha)$. Consequently, our result recovers Theorem J. This confirms that Theorem \ref{Th-4.1} is a generalization of J, leading to the following observations.
	\begin{enumerate}
		\item[(i)] If $\alpha=0$, the class $\mathcal{W}^0_{\mathcal{H}}(\alpha)$ reduces to $\mathcal{P}^0_{\mathcal{H}}$ and hence, from the identity $R_w(r)=L_w(1)$, we obtain the Bohr radius $R_f(0)\approx 0.285194$ (see \cite{Allu-BSM-2021}) for the class $ \mathcal{P}^0_{\mathcal{H}} $.\vspace{1.2mm}
		\item[(ii)] If $\alpha=1$, the class $\mathcal{W}^0_{\mathcal{H}}(\alpha)$ reduces to $\mathcal{W}^0_{\mathcal{H}}$ and we obtain the Bohr radius $R_f(1)\approx 0.58387765$ (see \cite{Allu-BSM-2021}) for the class $\mathcal{W}^0_{\mathcal{H}}$.\vspace{1.2mm}
		\item[(iii)] When the co-analytic part $g\equiv 0$, then $\mathcal{W}^0_{\mathcal{H}}(\alpha)$ reduces to $\mathcal{W}(\alpha)$. Therefore, from \cite[Lemma 1.9]{Allu-BSM-2021}, we observe that the Bohr radius for the class $\mathcal{W}(\alpha)$ is same as that of the class $\mathcal{W}^0_{\mathcal{H}}(\alpha)$.
	\end{enumerate}
\end{rem}
\subsection{\bf The Hurwitz--Lerch Zeta Function}
The Gaussian hypergeometric function, denoted by 
${}_2F_{1}(a,b;c;z)$, is defined by the power series  
\begin{align*}
	{}_2F_{1}(a,b;c;z)
	= \sum_{n=0}^{\infty} 
	\frac{(a)_n (b)_n}{(c)_n\, n!}\, z^n, 
	\qquad |z| < 1,
\end{align*}
where $(q)_n$ represents the \emph{Pochhammer symbol} (or rising factorial) defined by  
\begin{align*}
	(q)_n = q (q+1) (q+2) \cdots (q+n-1), \quad n \ge 1, 
	\qquad (q)_0 = 1.
\end{align*}
The function admits analytic continuation to the complex plane cut along $[1, \infty)$ and satisfies numerous transformation and differentiation formulas. In particular, the function 
${}_2F_{1}(1,b;b+1;z)$ admits the useful representation  
\begin{align*}
	{}_2F_{1}(1,b;b+1;z) = b\, \Phi(z,1,b),
\end{align*}
where $\Phi(z,s,a)$ denotes the \emph{Hurwitz--Lerch transcendent}, defined by  
\begin{align*}
	\Phi(z,s,a)
	:=\operatorname{HurwitzLerchPhi}(z,s,a)
	=\sum_{n=0}^{\infty}\frac{z^{n}}{(n+a)^{s}},
\end{align*}
which converges absolutely for $|z|<1$, and also for $|z|=1$ when $\Re(s)>1$.
It admits an analytic continuation to a large domain in $(z,s,a)\in\mathbb{C}^3$. A special case used in this paper is
\begin{align*}
	H(\alpha)
	:=\operatorname{HurwitzLerchPhi}\!\left(-1,1,1+\tfrac{1}{\alpha}\right)
	=\sum_{n=0}^{\infty}\frac{(-1)^{n}}{n+1+\tfrac{1}{\alpha}},
	\quad \alpha\in\mathbb{R}\setminus\{0,-1\}.
\end{align*}
For real $\alpha$ with $\alpha>0$ or $\alpha<-1$, the series in $H(\alpha)$
converges conditionally, since the terms form a positive decreasing sequence in magnitude and alternate in sign.  The function $H(\alpha)$ is real-valued on its domain, continuous, and strictly increasing for $\alpha>0$.

An alternative representation of $H(\alpha)$ can be given in terms of the
digamma function $\psi(z)=\Gamma'(z)/\Gamma(z)$:
\begin{align*}
	H(\alpha)
	=\frac{1}{2}\bigg(\psi\!\big(1+\tfrac{1}{2\alpha}\big)
	-\psi\!\big(\tfrac12+\tfrac{1}{2\alpha}\big)\bigg).
\end{align*}
This identity follows from the classical formula
\begin{align*}
	\Phi(-1,1,a)
	=\frac{1}{2}\bigg(\psi\!\big(\tfrac{a+1}{2}\big)
	-\psi\!\big(\tfrac{a}{2}\big)\bigg),
\end{align*}
which can be found in standard references such as \cite{GradRyzhik, Rudin-1976}.
\subsection{\bf Some consequences of Theorem \ref{Th-4.1}}
		Let $f=h+\overline{g}\in \mathcal{W}^{0}_{\mathcal{H}}(\alpha)$ be given by \eqref{e-7.2}  and suppose that $\{\varphi_n(r)\}^{\infty}_{n=0}\in\mathcal{F}$ with condition $\varphi_0(r)=1$ and $\varphi_n(r)=n^{\alpha}r^n\;(n\geq1)$,  we define the following functional
		\begin{align*}
			\mathcal{C}_{\alpha}^f(r)=r\varphi_0(r)+\sum_{n=2}^{\infty}\left(|a_n|+|b_n|\right)\varphi_n(r).
		\end{align*}
	
	Then, we have\\
\noindent{\bf (A):} $\mathcal{C}^f_{1}(r)\leq d\left(f(0), \partial f(\mathbb{D})\right)$ for $|z|=r\leq R_1(\alpha)$, where $R_1(\alpha)$ is the unique positive root in $(0,1)$ of the equation 
		\begin{align*}
			r+
			\frac{2r^{2}}{1 + a}\, {}_2F_{1}\!\left(1,\, 1 + \frac{1}{a};\, 2 + \frac{1}{a};\, r\right)
			=1+\dfrac{2(-1+H(\alpha)+\log 2)}{1-\alpha}
		\end{align*}. The number $R_1(\alpha)$ is best possible.\vspace{2mm}
		
\noindent{\bf (B):} $\mathcal{C}^f_{2}(r)\leq d\left(f(0), \partial f(\mathbb{D})\right)$ 
		for $|z|=r\leq R_2(\alpha)$, where $R_2(\alpha)$ is the unique positive root in $(0,1)$ of the equation 
		\begin{align*}
			r+2r^{2} \left( 
			\frac{2 - r}{(-1 + r)^{2}} 
			- \frac{2a\, {}_2F_{1}\!\left(3,\, 1 + \frac{1}{a};\, 2 + \frac{1}{a};\, r\right)}{1 + a} 
			\right)=1+\dfrac{2(-1+H(\alpha)+\log 2)}{1-\alpha}. 
			\end{align*}
			The number $R_2(\alpha)$ is best possible.\vspace{2mm}
		
\noindent{\bf (C):} $\mathcal{C}^f_{3}(r)\leq d\left(f(0), \partial f(\mathbb{D})\right)$ for $|z|=r\leq R_3(\alpha)$, where $R_3(\alpha)$ is the unique positive root in $(0,1)$ of the equation 
		\begin{align*}
			r&+2r^{2}\left(
			- \frac{4 - 3r + r^{2}}{(-1 + r)^{3}}
			- \frac{4a\, {}_2F_{1}\!\left(3,\, 1 + \frac{1}{a};\, 2 + \frac{1}{a};\, r\right)}{1 + a}
			- \frac{6a r\, {}_2F_{1}\!\left(4,\, 2 + \frac{1}{a};\, 3 + \frac{1}{a};\, r\right)}{1 + 2a}
			\right)\\&=1+\dfrac{2(-1+H(\alpha)+\log 2)}{1-\alpha}.
		\end{align*}
		The number $R_3(\alpha)$ is best possible.
\begin{proof}[\bf Proof of Theorem \ref{Th-4.1}]
	Since $ f\in \mathcal{W}^{0}_{\mathcal{H}}(\alpha)$, in view of Lemma H, we have
	\begin{align*}
		|f(z)|\geq|z|+\sum_{n=2}^{\infty}\frac{2(-1)^{n-1}|z|^n}{\alpha n^2+(1-\alpha)n}\;\;\mbox{for}\;\;|z|<1.
	\end{align*}
	Consequently, we have 
	\begin{align}\label{Eq-4.4}
		d(f(0), \partial f(\mathbb{D}))\geq 1+\sum_{n=2}^{\infty}\frac{2(-1)^{n-1}}{\alpha n^2+(1-\alpha)n}.
	\end{align}
	Thus, we see that
	\begin{align}\label{Eq-4.5}
		\liminf_{|z|\rightarrow1} |f(z)|\geq1+\sum_{n=2}^{\infty}\frac{2(-1)^{n-1}}{\alpha n^2+(1-\alpha)n}.
	\end{align}
	The Euclidean distance between $f(0)$ and the boundary of $f(\mathbb{D}$ is given by 
	\begin{align}\label{Eqn-4.6}
		d(f(0),\partial f(\mathbb{D}))=\liminf_{|z|\rightarrow1}|f(z)-f(0)|.
	\end{align}
	Since $f(0)=0$, and hence, we have $|f(z)-f(0)|=|f(z)|$. Therefore, from \eqref{Eq-4.5} and \eqref{Eqn-4.6}, we obtain
	\begin{align}\label{Eq-4.7}
		d(f(0),\partial f(\mathbb{D}))\geq 1+\sum_{n=2}^{\infty}\frac{2(-1)^{n-1}}{\alpha n^2+(1-\alpha)n}.
	\end{align}
	Let $L_{\alpha} : [0, 1]\to\mathbb{R}$ be defined by 
	\begin{align}\label{Eq-4.6}
		L_{\alpha}(r)=r\varphi_0(r)+2\sum_{n=2}^{\infty}\frac{\varphi_n(r)}{\alpha n^2+(1-\alpha)n}-1-\sum_{n=2}^{\infty}\frac{2(-1)^{n-1}}{\alpha n^2+(1-\alpha)n}.
	\end{align} 
	Clearly, $L_{\alpha}$ is continuous in $[0,1]$ and $L_{\alpha}$ is differentiable $(0,1)$.	Because, both the series
	\begin{align*}
		\sum_{n=2}^\infty \dfrac{\varphi_n(r)}{\alpha n^2+(1-\alpha)n}\; \mbox{and}\;\sum_{n=2}^\infty \dfrac{\varphi_n'(r)}{\alpha n^2+(1-\alpha)n}
	\end{align*}
		converge locally uniformly on $[0,1)$, we see that $L_\alpha$ is differentiable on $(0,1)$ with its derivative is 
	\begin{align*}
		L_\alpha'(r)
		= \varphi_0(r) + r\varphi_0'(r)
		+ 2\sum_{n=2}^\infty \frac{\varphi_n'(r)}{\alpha n^2+(1-\alpha)n},
	\end{align*}
		where the derivative-series converges locally uniformly on $(0,1)$.

		Fix an arbitrary compact sub-interval $[0,a]\subset(0,1)$. By the local uniform convergence hypotheses there exist constants $K,K_1>0$ (depending on $a$) such that for every $n\ge0$ and every $r\in[0,a]$, we have
		\begin{align*}
			|\varphi_n(r)|\le K,\qquad |\varphi_n'(r)|\le K_1.
		\end{align*}
		Since $\alpha>0$, the elementary bound
		\begin{align*}
			\alpha n^2+(1-\alpha)n \ge \alpha n^2,
		\end{align*}
		holds. Therefore, for all $r\in[0,a]$, we have
	\begin{align*}
		\left|\frac{\varphi_n(r)}{\alpha n^2+(1-\alpha)n}\right|
		\le \frac{K}{\alpha n^2}\; \mbox{and}\;
		\left|\frac{\varphi_n'(r)}{\alpha n^2+(1-\alpha)n}\right|
		\le \frac{K_1}{\alpha n^2}.
	\end{align*}
		Since the series $\sum_{n=2}^\infty \frac{K}{\alpha n^2}$ and $\sum_{n=2}^\infty \frac{K_1}{\alpha n^2}$ are convergent (they are constant multiples of the convergent $p$-series $\sum 1/n^2$), the Weierstrass $M$-test implies that both series of functions
		\begin{align*}
			\sum_{n=2}^\infty \frac{\varphi_n(r)}{\alpha n^2+(1-\alpha)n}
			\quad\text{and}\quad
			\sum_{n=2}^\infty \frac{\varphi_n'(r)}{\alpha n^2+(1-\alpha)n}
		\end{align*}
		converge uniformly on $[0,a]$. Because $[0,a]$ was an arbitrary compact subset of $(0,1)$, the convergence is locally uniform on $[0,1)$.\vspace{1.2mm}
		
		Each term $\frac{\varphi_n(r)}{\alpha n^2+(1-\alpha)n}$ is differentiable on $(0,1)$ with derivative $\frac{\varphi_n'(r)}{\alpha n^2+(1-\alpha)n}$. By the uniform convergence of the series and the series of its derivatives on $[0,a]$, we may apply the standard theorem on term-by-term differentiation \cite[Theorem 7.17]{Rudin-1976} to conclude that the sum is differentiable on $(0,a)$, and its derivative equals the termwise sum of the derivatives. Since $a \in (0,1)$ was arbitrary, the same holds on $(0,1)$.\vspace{1.2mm}
		
			Moreover, since $\varphi_n(0)=0$ for all $n\in\mathbb{N}$, it follows that
	\begin{align*}
		L_{\alpha}(0)&=2\sum_{n=2}^{\infty}\frac{\varphi_n(0)}{\alpha n^2+(1-\alpha)n}-1-\sum_{n=2}^{\infty}\frac{2(-1)^{n-1}}{\alpha n^2+(1-\alpha)n}\\&=-1-\sum_{n=2}^{\infty}\frac{2(-1)^{n-1}}{\alpha n^2+(1-\alpha)n}<0.
	\end{align*}
	Since
		\begin{align*}
			H(\alpha)=\operatorname{HurwitzLerchPhi}\!\left(-1,1,1+\tfrac{1}{\alpha}\right)
			=\sum_{n=0}^{\infty}\frac{(-1)^{n}}{n+1+\tfrac{1}{\alpha}},
		\end{align*}
		and define
		\begin{align*}
				G(\alpha)=H(\alpha)+\log 2-\alpha,\; \mbox{for}\; 0<\alpha<1.
		\end{align*}
		Since $1-\alpha>0$, the given inequality is equivalent to
		\begin{align*}
			H(\alpha)+\log 2>\alpha,
		\end{align*}
		that is, $G(\alpha)>0$ on $(0,1)$.
		
		For $a=1+\tfrac{1}{\alpha}>2$, we have
		\begin{align*}
				H^{\prime}(\alpha)
			=\frac{1}{\alpha^{2}}\sum_{n=0}^{\infty}\frac{(-1)^{n}}{(n+a)^{2}}.
		\end{align*}
		Because the sequence $\{(n+a)^{-2}\}$ is positive and decreasing, the alternating-series estimate gives
		\begin{align*}
			0<\sum_{n=0}^{\infty}\frac{(-1)^{n}}{(n+a)^{2}}<\frac{1}{a^{2}}.
		\end{align*}
		Hence,
		\begin{align*}
			0<H^{\prime}(\alpha)<\frac{1}{\alpha^{2}a^{2}}
			=\frac{1}{(1+\alpha)^{2}}<1,
		\end{align*}
		which implies that $G'(\alpha)=H'(\alpha)-1<0$.
		Thus, $G(\alpha)$ is strictly decreasing on $(0,1)$. Moreover,
		\begin{align*}
			\lim_{\alpha\to0^{+}}G(\alpha)=\log 2>0,\;
			\lim_{\alpha\to1^{-}}G(\alpha)
			=H(1^{-})+\log 2-1
			=(1-\log 2)+\log 2-1=0.
		\end{align*}
		Since $G$ is continuous and strictly decreasing from $\log 2$ to $0$,
		it follows that $G(\alpha)>0$ for every $0<\alpha<1$.
		Consequently,
		\begin{align*}
			-1-\frac{-1+H(\alpha)+\log 2}{1-\alpha}<0
			\qquad (0<\alpha<1).
		\end{align*}
	Since $\varphi_n(1)=1$ for all $n\in\mathbb{N}\cup\{0\}$, we have
	\begin{align*}
		L_{\alpha}(1)&=\varphi_0(1)+2\sum_{n=2}^{\infty}\frac{\varphi_n(1)}{\alpha n^2+(1-\alpha)n}-1-\sum_{n=2}^{\infty}\frac{2(-1)^{n-1}}{\alpha n^2+(1-\alpha)n}\\&=2\bigg(\sum_{n=2}^{\infty}\frac{1}{\alpha n^2+(1-\alpha)n}-\dfrac{-1+H(\alpha)+\log 2}{1-\alpha}\bigg)>0.
	\end{align*}
where
	\begin{align*}
		\sum_{n=2}^{\infty}\frac{(-1)^{n-1}}{\alpha n^2+(1-\alpha)n}=\dfrac{-1+H(\alpha)+\log 2}{1-\alpha}<0\,\;\mbox{for}\; 0\leq\alpha<1.
	\end{align*}
	By the \textit{Intermediate Value Theorem}, the function $H_{\alpha}$ has a root in $(0,1)$, $R_f(M)$ say. To show that $R_f(M)$ is unique, it is sufficient to show that $H_M$ is a monotonic function on $[0, 1]$. Since  the set of all sequences $\{\varphi_n(r)\}^{\infty}_n$ of non-negative differentiable functions in $[0,1)$ such that the series $\sum_{n=0}^{\infty}\varphi_n(r)$ converges locally uniformly on the interval $[0,1)$, $\;\forall\; n\in\mathbb{N}\cup \{0\}$, an easy computation  shows that
			\begin{align*}
			L_\alpha'(r)
			= \varphi_0(r) + r\varphi_0'(r)
			+ 2\sum_{n=2}^\infty \frac{\varphi_n'(r)}{\alpha n^2+(1-\alpha)n}\;\;\mbox{for all}\;\;r\in(0,1).
		\end{align*}
		Thus, we conclude that $L_{\alpha}$ is a strictly monotone increasing function on $(0,1)$. 
		Then, $L_{\alpha}(R_{f}(\alpha))=0$ leads to 
		\begin{align}\label{EQ-4.6}
		R_{f}(\alpha)\varphi_0(R_{f}(\alpha))+2\sum_{n=2}^{\infty}\frac{\varphi_n(R_{f}(\alpha))}{\alpha n^2+(1-\alpha)n}=1+\sum_{n=2}^{\infty}\frac{2(-1)^{n-1}}{\alpha n^2+(1-\alpha)n}.
		\end{align}
		The Lemma G in view of \eqref{Eq-4.4} yields that
		\begin{align*}
			B_f(r)&=r\varphi_0(r)+\sum_{n=2}^{\infty}\left(|a_n|+|b_n|\right)\\&\leq r\varphi_0(r)+2\sum_{n=2}^{\infty}\frac{\varphi_n(r)}{\alpha n^2+(1-\alpha)n}\\&\leq 1+\sum_{n=2}^{\infty}\frac{2(-1)^{n-1}}{\alpha n^2+(1-\alpha)n}\\&\leq d(f(0), \partial f(\mathbb{D}))
		\end{align*}
		for $|z|=r\leq R_f(\alpha)$. Thus the desired inequality is established.\vspace{1.2mm}
		
		Next we have to prove that the number $R_f(\alpha)$ is best possible.
		Henceforth, we consider the function $f_{\alpha}$ given by 
		\begin{align*}
			f_{\alpha}(z)=z+2\sum_{n=2}^{\infty}\frac{z^n}{\alpha n^2+(1-\alpha)n}.
		\end{align*}
		Clearly, we see that $f_{\alpha}\in \mathcal{W}^{0}_{\mathcal{H}}(\alpha)$. For $r> R_f(\alpha)$ and $f=f_{\alpha}$, an easy computation in view of part-(i) of Lemmas G and H shows that
			\begin{align*}
			B_{f_{\alpha}}(r)&=r\varphi_0(r)+\sum_{n=2}^{\infty}\left(|a_n|+|b_n|\right)\varphi_n(r)\\&= r\varphi_0(r)+2\sum_{n=2}^{\infty}\frac{\varphi_n(r)}{\alpha n^2+(1-\alpha)n}\\&>	R_{f}(\alpha)\varphi_0(R_{f}(\alpha))+2\sum_{n=2}^{\infty}\frac{\varphi_n(R_{f}(\alpha))}{\alpha n^2+(1-\alpha)n}\\&= 1+\sum_{n=2}^{\infty}\frac{2(-1)^{n-1}}{\alpha n^2+(1-\alpha)n}\\&= d(f_{\alpha}(0), \partial f_{\alpha}(\mathbb{D}))
		\end{align*}
		This shows that the number $R_f(\alpha)$ is best possible.
\end{proof}

\noindent{\bf Acknowledgment:} The authors wish to thank the anonymous referees for their detailed comments and valuable suggestions, which have significantly improved the presentation of the paper.The first author is supported by SERB (File No.: SUR/2022/002244, Govt. of India), and the second author is supported by UGC-JRF (NTA Ref. No.: $211610135410$), New Delhi, India.
\vspace{3mm}

\noindent\textbf{Compliance of Ethical Standards.}\\

\noindent\textbf{Conflict of interest.} The authors declare that there is no conflict  of interest regarding the publication of this paper.\vspace{1.5mm}

\noindent\textbf{Data availability statement.}  Data sharing not applicable to this article as no datasets were generated or analyzed during the current study.

\end{document}